\documentclass[11pt, twoside, letterpaper]{amsart}

\usepackage{lipsum}
\usepackage{t1enc}
\usepackage{latexsym}
\usepackage{amssymb}
\usepackage{graphicx}
\usepackage{amsmath}
\usepackage{amsthm}
\usepackage{amsfonts}
\usepackage{comment}
\usepackage{tikz-cd}
\usepackage{tikz}

\usepackage{mathrsfs}

\usepackage[all]{xy}
\usepackage[british]{babel}

\usepackage{color}
\usepackage[hypertexnames=false,
    pdftex,
	pdfpagemode=UseNone,
	breaklinks=true,
	extension=pdf,
	colorlinks=true,
	linkcolor=blue,
	citecolor=blue,
	urlcolor=blue,
]{hyperref}

\newcommand{\cyclic}{\mathop{\kern0.9ex{{+}\kern-2.10ex\raise-0.20
      ex\hbox{\Large\hbox{$\circlearrowright$}}}}\limits}

\newtheoremstyle{daniel}{3.0mm}{2mm}{\itshape}{}{\bfseries}{.}{1.5mm}{}
\theoremstyle{daniel}
\newtheorem{thm}{Theorem}[section]
\newtheorem{prop}[thm]{Proposition}

\newtheorem{cor}[thm]{Corollary}
\newtheorem{Exs}[thm]{Examples}
\newtheorem{Rems}[thm]{Remarks}

\newtheorem*{thm*}{Theorem}
\newtheorem*{cor*}{Corollary}
\newtheorem*{thm3.6}{Theorem 3.6}
\newtheorem*{thm4.3}{Theorem 4.3}

\newtheorem*{prop*}{Proposition}
\newtheorem*{Notation}{Notation}

\newtheorem{Lem}[thm]{Lemma}

\newtheorem{Def}[thm]{Definition}

\newtheorem{Rem}[thm]{Remark}

\newtheorem{Ex}[thm]{Example}

\newtheorem*{Setup}{Setup}

\newenvironment{rem}   {\begin{Rem}\em}{\end{Rem}}

\newenvironment{notation}   {\begin{Notation}\em}{\end{Notation}}

\def\cC{\mathcal C }

\def\cE{\mathcal E}
\def\cF{\mathcal F}
\def\cH{\mathcal H}

\def\cO{\mathcal O}

\def\cT{\mathcal T}
\def\cK{\mathcal K}

\def\gE{{\mathfrak E}}

\renewcommand{\P}{\mathbb{P}}
\newcommand{\Q}{\mathbb{Q}}
\newcommand{\R}{\mathbb{R}}

\def\codim{\mbox{codim}}
\DeclareMathOperator{\car}{char}

\def\ch{\mbox{ch}}

\DeclareMathOperator{\Ext}{Ext}

\def\Hom{\mbox{Hom}}

\DeclareMathOperator{\mult}{mult}

\DeclareMathOperator{\SchSupp}{SchSupp}

\def\Supp{\mbox{Supp}}
\def\Todd{\mbox{Todd}}

\DeclareMathOperator{\depth}{depth}

\DeclareMathOperator{\Amp}{Amp}
\DeclareMathOperator{\Pseff}{Pseff}
\DeclareMathOperator{\Nef}{Nef}

\DeclareMathOperator{\ev}{ev}
\DeclareMathOperator{\num}{num}
\DeclareMathOperator{\rk}{rk}
\DeclareMathOperator{\BSS}{\mathbf{UB}}

\numberwithin{equation}{section}

\begin{document}
\title[Uniform boundedness of semistable pure sheaves]{Uniform boundedness of semistable pure sheaves on projective manifolds}

\author{Mihai Pavel, Julius Ross, Matei Toma}

\address{Institute of Mathematics of the Romanian Academy,
P.O. Box 1-764, 014700 Bucharest, Romania
}
\email{cpavel@imar.ro}
\address{Department of Mathematics, Statistics, and Computer Science, University of Illinois at Chicago, 322 Science and Engineering Offices (M/C 249), 851 S. Morgan Street, Chicago, IL 60607
 }
\email{juliusro@uic.edu}
 
\address{ Universit\'e de Lorraine, CNRS, IECL, F-54000 Nancy, France
}

\email{Matei.Toma@univ-lorraine.fr}

\date{\today}

\begin{abstract} We prove uniform boundedness statements for semistable pure sheaves on projective manifolds.  For example, we prove that the set of isomorphism classes of pure sheaves of dimension 2 that are slope semistable with respect to ample classes that vary in a compact set $K$ are bounded.  We also prove uniform boundedness for pure sheaves of higher dimension, but with restrictions on the compact set $K$.  As applications we get new statements about moduli spaces of semistable sheaves, and the wall-chamber structure that governs their variation.
\end{abstract}
\maketitle
\setcounter{tocdepth}{1}
\tableofcontents
\noindent


\section{Introduction}
When constructing moduli spaces of objects in algebraic geometry one of the most basic properties desired is \emph{boundedness} as this ensures that such moduli spaces will have certain finiteness properties (e.g.\ be of finite type) and, typically, ensuring such boundedness requires imposing some choice of stability condition.  

Once such moduli spaces have been constructed, it is natural to ask how they depend on this choice of stability.   In the best cases one can divide the space of stability choices into a (locally) finite number of chambers separated by walls such that the moduli spaces are unchanged within chambers, and undergo some transformation, or \emph{wall-crossing} as the choice of stability passes from one chamber to another.  

Typically obtaining such a finite chamber structure requires stronger bound\-edness results.   That is, one requires not just that the set of objects that are semistable with respect to one stability condition is bounded, but that this boundedness remains true uniformly as the stability condition varies within a compact set.

In this paper we address this stronger boundedness in the case of pure sheaves on projective manifolds.  The precise statements require some definitions (see Section \ref{sec:preparations} for more details).  On a projective manifold $X$ of dimension $n$ {over an algebraically closed field} $k$ we let $\Amp^{i}(X)$  denote the interior of the nef cone of $i$-codimensional cycles. Given $$\alpha_{d-1}\in \Amp^{d-1}(X)\text{ and }\alpha_d\in \Amp^d(X)$$ we define the slope of a purely $d$-dimensional sheaf $E$ on $X$ to be 
\[
    \mu^{\alpha_{d-1}}_{\alpha_d}(E):= \frac{\deg_{\alpha_{d-1}}(E)}{\deg_{\alpha_d}(E)},
\]
where the degree of $E$ with respect to a class $\beta \in \Amp^p(X)$ is $$\deg_{\beta}(E) : = \int_X \ch(E)\beta \Todd_X.$$
 When $h \in \Amp^1(X)$ is an ample class, this gives the classical slope function upon setting $(\alpha_d,\alpha_{d-1}) = (h^d, h^{d-1})$.

\begin{Def}[Semistability]
    We say that a $d$-dimensional sheaf $E$ on $X$ is {\em  semistable with respect to} $(\beta,\alpha)$ if it is pure 
    and for all pure $d$-dimensional quotients $E \to F$ we have
\[
  \mu^{\alpha}_\beta(F) \ge \mu^{\alpha}_\beta(E).
\]
\end{Def}

We recall that a set $\mathcal S$ of isomorphism classes of sheaves on $X$ is said to be \emph{bounded} if there is a scheme $S$ of finite type and a coherent $\mathcal O_{X\times S}$ sheaf $\mathcal E$ on $X\times S$ such that every sheaf in $\mathcal S$ is isomorphic to $\mathcal E|_{X\times \{s\}}$ for some $s\in S$ \cite[Definition 1.7.5]{HL}.

\begin{Def}[Uniform Boundedness]
Given a numerical class $\gamma \in K(X)_\text{num}$ of dimension $d > 0$ and a compact subset $K \subset \Amp^{d}(X) \times \Amp^{d-1}(X)$, by \emph{uniform boundedness} $\BSS$ we mean the following statement:

\vspace{0.5em}
$\BSS(K,\gamma)$:  The set of  isomorphism classes of coherent sheaves $E$ of numerical type $\gamma$ on $X$ that are semistable with respect to some $(\beta,\alpha)$ in $K$ is \textit{bounded}. 
\end{Def}

As discussed above, the case of a singleton $K = \{(\beta,\alpha)\}$ is important to ensure the existence of \textit{finite-type} moduli stacks of pure sheaves, and provides one of the key ingredients in the construction of proper moduli spaces of sheaves. On the other hand, for studying the variation of such moduli spaces with respect to the change of semistability, it is crucial to allow $(\beta,\alpha)$ to vary uniformly in a given compact set $K$. The statement $\BSS$ is already known for one-dimensional pure sheaves (see \cite[Proposition 3.6]{GRT2preprint} and \cite[Proposition 7.19]{joyce2021enumerative}). The first open case we treat here is that of two-dimensional pure sheaves, for which we prove the following general result:

\begin{thm}[$\subset$ Corollary \ref{cor:UniformBoundednessD=2}, Uniform boundedness for pure 2-dimensional sheaves]\label{thm:introD=2}
The statement $\BSS(K,\gamma)$ holds for arbitrary $2$-dimensional numerical classes $\gamma$ and any compact $K \subset \Amp^2(X) \times \Amp^1(X)$. 
\end{thm}

Next we consider uniform boundedness for higher dimensional pure sheaves. Our approach involves the use of a generalized Grauert-Mülich-Spindler Theorem (see Section \ref{section:GrauertMuelich}). We obtain the following uniform boundedness result:

\begin{thm}[= Corollary \ref{cor:higherDimension}, Uniform boundedness of pure sheaves]\label{thm:boundednessHD:intro}
Assume $\car(k)=0$, let $d \ge 2$ and let $h_1,\ldots,h_{d-2} \in \Amp^1(X)$ be rational ample divisor classes.  Suppose also $K_1 \subset \Amp^1(X)$ and $K_d \subset \Amp^d(X)$ are compact sets, and define 
 $$K_{d-1} = K_1 h_1 \cdots h_{d-2} \subset \Amp^{d-1}(X).$$
Then the statement $\BSS(K_d \times K_{d-1},\gamma)$ holds for any $d$-dimensional numerical class $\gamma$.
\end{thm}

This allows us connect two ``complete intersection'' stability conditions by a compact set (in fact a collection of segments) and prove uniform boundedness with respect to this compact set (see Figure \ref{fig:linesofboundedstabilityparameters} for an illustration in the case $d=4$).

\begin{cor}[= Corollary \ref{cor:boundednessSegments}]\label{cor:boundednessSegments:intro}
Assume $\car(k)=0$, let $\gamma$ be a numerical class of $d$-dimensional coherent sheaves on $X$, let $\alpha, \alpha' \in \Amp^1(X)$ and let $h_1,\ldots,h_{d-2},h_1',\ldots,h'_{d-2} \in \Amp^1(X)$ be rational ample classes. Then there is a connected compact subset $K_{d-1} \subset \Amp^{d-1}(X)$ containing $\alpha h_1\cdots h_{d-2}$ and $\alpha' h'_1\cdots h'_{d-2}$ such that for any non-empty compact set $K_d \subset \Amp^{d}(X)$, the statement 
$\BSS(K_{d} \times K_{d-1} ,\gamma)$
is true. 
\end{cor}

\begin{figure}
\centering
\begin{tikzpicture}[x=1cm,y=1cm]
\fill (0,0) circle(2pt);
\fill (2,-2) circle(2pt);
\fill (4,1) circle(2pt);
\fill (6,-2) circle(2pt);
\fill (8,0) circle(2pt);
\draw [thick] (0,0)--(2,-2);
\draw [thick] (2,-2)--(4,1);
\draw [thick] (6,-2)--(4,1);
\draw [thick] (6,-2)--(8,0);
\node[above] at (0,0){$\alpha h_1 h_2$};
\node[below] at (2,-2){$h h_1 h_2$};
\node[above] at (4,1){$h h_1' h_2$};
\node[below] at (6,-2){$h h_1' h_2'$};
\node[above] at (8,0){$\alpha' h_1' h_2'$};

\end{tikzpicture}
\caption{Illustration of bounded stability parameter set in $\Amp^3(X)$ joining $\alpha h_1h_2$ and $\alpha' h_1'h_2'$}
\label{fig:linesofboundedstabilityparameters}
\end{figure}
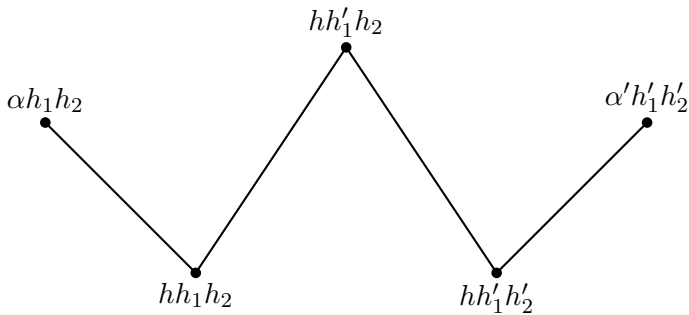

As per the original motivation, these boundedness results allow for the construction of (new) moduli spaces of semistable pure sheaves, and a finite chamber structure that describes how these moduli spaces change as this semistability condition varies.  We refer the reader to Section \ref{sec:corollaries} for details.

\subsection*{Comparison with Other Work: } 
Boundedness of semistable sheaves was intensively studied in the context of classical slope-semistability, and it took the efforts of many mathematicians to completely settle the case $K = \{(h^d, h^{d-1})\}$ with $h \in \Amp^1(X)$ rational (e.g. \cite{atiyah57vector,KleimanSga6, takemoto1972stable, Maruyama81boundedness, Gieseker, Langer}).

Statements about uniform boundedness are more recent, and were obtained for torsion-free sheaves on smooth surfaces by Matsuki-Wentworth \cite{MatsukiWentworth} and for torsion-free sheaves on higher dimensional base by Greb-Ross-Toma (see \cite{GrebToma, GRT1}). In both cases the approach is based on two main ingredients: the Bogomolov-Gieseker inequality and the Hodge Index Theorem, which are both valid when the underlying base space is smooth. To date, it remains uncertain whether these techniques are applicable in the general setting of pure sheaves, whose support scheme might have very bad singularities (see \cite{wu2023bogomolovs, langer2023bridgeland} for recent progress in this direction).

Instead, we adopt a different strategy to deal with the uniform boundedness problem in pure dimension two, inspired by Maruyama's work \cite{Maruyama75boundedness, Maruyama81boundedness}. This way we establish the two-dimensional $\BSS$ hypothesis in its broadest scope, valid in arbitrary characteristic. 

These results have applications to enumerative geometry through Joyce's work \cite{joyce2021enumerative}, in particular to counts of pure two-dimensional sheaves on a Fano threefold.  Specifically Corollary \ref{cor:UniformBoundednessD=2} gives a positive answer to \cite[Problem 7.21]{joyce2021enumerative}.

Our approach in higher dimensions is based on an analogue of the Grauert-Mülich-Spindler Theorem, which generalizes the classical result (see \cite[Theorem 3.1.2]{HL}) to the case of pure sheaves. This is likely to be of independent interest, with the caveat that it only holds in zero characteristic. We emphasise that we obtain here the first uniform boundedness results for pure sheaves in higher dimensions. It remains open if our results are also valid in positive characteristic. 

As already noted, several applications of our results to the moduli theory of sheaves are presented in Section \ref{sec:corollaries}. They are closely related to the work in \cite{MegyPavelToma}, which proposes a general framework for studying wall-crossing phenomena in higher dimensions.

\subsection*{Acknowlegements: } JR is partially supported by the National Science Foundation under Grant No. DMS-1749447. MP and MT acknowledge financial support from  IRN ECO-Maths. MP was also partly supported by the PNRR grant CF 44/14.11.2022 \textit{Cohomological Hall algebras
of smooth surfaces and applications}.

\section{Preparations}\label{sec:preparations}
We start by fixing some notations and conventions. Throughout the paper $X$ will be a smooth projective variety of dimension $n$ over an algebraically closed field.

For $0\le p\le n$ we will denote by $N^p(X)_\R$ the numerical group of real codimension $p$ cycles on $X$, by $\Nef^p(X)$ the cone in $N^p(X)_\R$ dual to the pseudoeffective cone $\Pseff^{n-p}(X)$,   and by $\Amp^{p}(X)$ the interior of $\Nef^p(X)$. 
The cones $\Amp^p(X)$ are non-empty for each $p\in\{0,\ldots,n\}$.
The elements of $\Amp^{p}(X)$ will be called {\em ample $p$-classes}. For  $\alpha \in \Amp^p(X)$ we define a degree function
\[ 
	\deg_{\alpha} : K(X) \to \R, \ [F]\mapsto \int_X \ch(F)\alpha \Todd_X,
\]
where $K(X)$ is the Grothendieck group of coherent sheaves on 
$X$.

We also fix a very ample line bundle $\mathcal O_X(1)$ on $X$ and put $h\in \Amp^{1}(X)$  the class of a  member of the linear system $|\mathcal O_X(1)|$.

Let $\alpha_{d} \in \Amp^{d}(X)$ and $\alpha_{d-1} \in \Amp^{d-1}(X)$. Let $E$ be a purely $d$-dimensional coherent sheaf on $X$. We define its $(\alpha_d,\alpha_{d-1})$-slope by
\[
    \mu^{\alpha_{d-1}}_{\alpha_d}(E):= \frac{\deg_{\alpha_{d-1}}(E)}{\deg_{\alpha_d}(E)}.
\]
For simplicity, we write $\mu^{\alpha_{d-1}} := \mu^{\alpha_{d-1}}_{h^d}$ and $\mu := \mu^{h^{d-1}}_{h^d}$. 

For any coherent sheaf $E$ on $X$, we set 
\[
    \theta_j(E) = \sum_{i=j}^n \ch_{n-i}(E)\Todd_{i-j}(X)
\]
for $0 \le j \le n$. We will consider the $\theta_j(E)$ as numerical classes, i.e. in the groups $N^{n-j}(X)_\Q$. Note that $\theta_j(E) = 0$ if $j > \dim(E)$ and that $\theta_j(E)$ is the class a positive  $j$-dimensional cycle if $j = \dim(E)$ (and $E\ne0$). With this notation, we see that for $\omega \in \Amp^p(X)$ 
\[
    \deg_{\omega} = \int_X \omega \cdot \theta_p.
\]
In particular if $p=\dim(E)\ge0$, we have $\deg_{\omega}(E)>0$.

By the {\em numerical class} (or {\em numerical type})  of a coherent sheaf $E$ on $X$ we understand its class in the numerical Grothendieck group $K(X)_{\num}$ of coherent sheaves on $X$, cf. \cite[Section 8.1]{HL}. It is known that the Chern character gives an isomorphism $\ch: K(X)_\Q\to A(X)_\Q$ between the rational Grothendieck group and the rational Chow group, \cite[Example 15.2.16]{FultonIT}, and that this isomorphism passes to the numerical groups $\ch:K(X)_{\Q \num}\to N^*(X)_\Q$, \cite[Example 19.1.5]{FultonIT}. Thus the numerical class of a coherent sheaf is perfectly determined by its numerical Chern character and thus also by its $\theta_j$-classes.

If $F$ is a coherent subsheaf of a coherent sheaf $E$ on $X$ we recall that the {\it saturation of $F$ in $E$} is the minimal coherent subsheaf $F'$ of $E$ containing $F$ such that $E/F'$ is zero or pure of dimension $\dim E$, cf. \cite[Definition 1.1.5]{HL}.

If $E$ is a coherent sheaf  on $X$ and $s$ is a positive integer, we shall denote by $N_s(E)$ the maximal coherent subsheaf of $E$ of dimension less than $s$ and write  
    $E_{(s)}:=E/N_s(E)$.

We now introduce a semistability notion with respect to classes which are allowed to vary in a compact set. 

\begin{Def}\label{def:stability} 
    Let $K \subset \Amp^{d}(X) \times \Amp^{d-1}(X)$ be compact, and let $a \in \R$. We say that a $d$-dimensional sheaf $E$ on $X$ is {\em  $(a,K)$-semistable} if it is pure and there is some $(\beta,\alpha) \in K$ such that for all pure $d$-dimensional quotients $E \to F$ we have
\[
  \mu^{\alpha}_\beta(F) \ge \mu^{\alpha}_\beta(E) - a.
\]
For $(\beta,\alpha) \in\Amp^{d}(X) \times \Amp^{d-1}(X)$  we shall simply  say that $E$ is {\em  $(a,\beta,\alpha)$-semistable} in this case. Coherent sheaves on $X$ 
which are $(0,\beta,\alpha)$-semistable will be also called {\em  $(\beta,\alpha)$-slope-semistable} or {\em slope-semistable with respect to $(\beta,\alpha)$}.
\end{Def}

\begin{Lem}\label{lem:changingDegs}
Let $K = K_{d} \times K_{d-1}$, $K' = K_{d}'\times K_{d-1}$, where $K_{d-1}$ is a compact subset of $\Amp^{d-1}(X)$ and $K_d, K'_d$ are compact subsets of $\Amp^d(X)$. Fix $a \in \R$ and $\theta_d \in N^{n-d}(X)_\Q$, $\theta_{d-1} \in N^{n-d+1}(X)_\Q$. There exists some $a' \in \R$ such that as soon as $E$ is an $(a,K)$-semistable $d$-dimensional sheaf on $X$  with $\theta_d(E) =\theta_d$ and $\theta_{d-1}(E) = \theta_{d-1}$, then $E$ is also  $(a',K')$-semistable.
\end{Lem}
\begin{proof}
Let $E$ be a purely $d$-dimensional sheaf as in the statement, and consider a pure quotient $E \to F$. As $E$ is $(a,K)$-semistable, there is some $(\beta,\alpha) \in K$ such that
\[
  \mu^{\alpha}_\beta(F) \ge \mu^{\alpha}_\beta(E) - a.
\]
Now for $(\beta',\alpha) \in K'$ we have
\begin{align*}
   \mu^{\alpha}_{\beta'}(F) ={}&  \mu^\alpha_{\beta}(F) \frac{\deg_{\beta}(F)}{\deg_{\beta'}(F)} \\
     \ge{}& \mu^\alpha_{\beta}(E) \frac{\deg_{\beta}(F)}{\deg_{\beta'}(F)} - a\frac{\deg_{\beta}(F)}{\deg_{\beta'}(F)}\\
     ={}& \mu^\alpha_{\beta'}(E) \frac{\deg_{\beta'}(E)}{\deg_{\beta}(E)} \frac{\deg_{\beta}(F)}{\deg_{\beta'}(F)} - a\frac{\deg_{\beta}(F)}{\deg_{\beta'}(F)} \\
     ={}& \mu^\alpha_{\beta'}(E) - \mu^\alpha_{\beta'}(E)\left(1 - \frac{\deg_{\beta'}(E)}{\deg_{\beta}(E)} \frac{\deg_{\beta}(F)}{\deg_{\beta'}(F)} \right) - a\frac{\deg_{\beta}(F)}{\deg_{\beta'}(F)}.
\end{align*}
 Since $F$ is a quotient of $E$ we have $\deg_{\beta}(F)\le \deg_{\beta}(E)$.   So clearly the quantity
\[
\mu^\alpha_{\beta'}(E)\left(1 - \frac{\deg_{\beta'}(E)}{\deg_{\beta}(E)} \frac{\deg_{\beta}(F)}{\deg_{\beta'}(F)} \right) + a\frac{\deg_{\beta}(F)}{\deg_{\beta'}(F)}.
\]
is bounded from above by some $a' \in \R$ depending on $K, K', \theta_{d-1}, \theta_d, a$, and independent of $E \to F$.
\end{proof}

\begin{Def}[Bounded Stability Parameter]\label{def:BoundedStabilitySet}
 Let $\gamma$ be a fixed numerical class of $d$-dimensional coherent sheaves on $X$. We say that a compact set $K \subset \Amp^{d}(X) \times \Amp^{d-1}(X)$ is a {\em  bounded stability parameter set for $\gamma$} if for any $a \in \R$ the set of isomorphism classes of $d$-dimensional  $(a,K)$-semistable sheaves of class $\gamma$ on $X$ is bounded.   
\end{Def}
\begin{rem}\label{rem:compacta}
We will typically apply Lemma \ref{lem:changingDegs} to the situation where $K = K_{d} \times K_{d-1}$ is arbitrary and $K' = \{h^d\}\times K_{d-1}$. Thus in order to prove that $K$ is a bounded stability parameter set for a numerical class $\gamma$ of $d$-dimensional coherent sheaves on  $X$ it will be enough to do so for $K'$.
\end{rem}

We will use the next elementary "convexity" result when changing degree functions in our boundedness statements.

\begin{Lem}\label{lem:convexity}
 Let $V$ be a real vector space of dimension $m$, let  $e_1,\ldots,e_m$ be a basis of $V$  and let $K$ be a compact set contained in the interior of the convex cone spanned by  $e_1,\ldots,e_m$ in $V$.  Then for any choice of real constants $c_1,\ldots, c_m,c$, the subset $$L_{c_1,\ldots, c_m,c,K}:=\{ f\in V^* \ | \ f(e_j)\ge c_j \ \forall j, \ \exists v\in K \ f(v)\le c\}$$ of the dual space $V^*$ is compact. 
\end{Lem}
\begin{proof}
If we write elements $v\in K$ as $v=\sum_{j=1}^m a_je_j\in V$ with $a_j>0$ for $j=1,\ldots,m$ and we express an element $f$ of $L_{c_1,\ldots, c_m,c,K}$ in terms of the dual basis as $f=\sum_{j=1}^m b_je_j^*$, we immediately obtain bounds
$$c_j\le b_j\le \frac{c-\sum_{i\ne j}^m a_ic_i}{a_j}$$
for each coefficient $b_j$ of $f$.
\end{proof}

\section{The case of two-dimensional support}\label{section:dimension2}

As before $X$ will denote a smooth projective variety over an algebraically closed field $k$ of arbitrary characteristic and $h = [\cO_{X}(1)] \in \Amp^1(X)$  the class of a very ample divisor. For a two-dimensional coherent sheaf $E$ on $X$, its Hilbert polynomial with respect to any ample class $\alpha \in \Amp^{1}(X)$ is defined as
$$P^{\alpha}_{E}(m) = \int_{X} \ch(E)\Todd(X) e^{m\alpha},$$ which may be rewritten as

\begin{align*}
P^{\alpha}_{E}(m) ={}& \frac{m^{2}}{2}\alpha^{2}\theta_2(E) + m\alpha \theta_1(E) + \theta_0(E),
\end{align*}
where
\begin{align*}
    \theta_2(E) &= \ch_{n-2}(E)\Todd_{0}(X),\\
    \theta_1(E) &= \ch_{n-2}(E)\Todd_{1}(X) + \ch_{n-1}(E)\Todd_{0}(X),\\
    \theta_0(E) &= \ch_{n-2}(E)\Todd_{2}(X) + \ch_{n-1}(E)\Todd_{1}(X) + \ch_{n}(E)\Todd_{0}(X)
\end{align*}
are classes in $N^*(X)$. In particular $\theta_1$ is the dual class of the homology Todd class $\tau_{1}(E) \in N_1(X)$, see \cite[Section 18.2]{FultonIT}.

We will call $\mult(E):=h^{2}\theta_2(E)$ and $\deg_\alpha(E):=\alpha\theta_1(E)$ the {\em multiplicity of $E$}, respectively the {\em degree of $E$ with respect to $\alpha$} and will define the {\em slope of $E$ with respect to $\alpha$} by 
$$\mu^\alpha(E):=\frac{\deg_\alpha(E)}{\mult(E)}.$$ (We remark that this definition slightly differs from other conventions.)

If the scheme-theoretic support of $E$ is an integral surface $Y$ we will denote by $r(E)$ the rank of $E$ seen as an $\cO_Y$-module. In this case we have $\mult(E)=r(E)\mult(\cO_Y)$. (See \cite[Definition 05JV]{stacks-project} for the definition of the scheme theoretic support.)

\begin{Lem}\label{lem:rank1inequality}Let $L$ be a pure $2$-dimensional coherent sheaf whose scheme-theoretic support is an integral surface $Y$. Suppose moreover that $L$ has rank one on its support. Then $\chi(L)$ is upper bounded in terms of $\mult(L)=h^{2}\theta_2(L)=\deg_{h^{2}}(Y)$ and $\deg_h(L)=h\theta_1(L)$. 
\end{Lem}
\begin{proof}
Consider a linear projection $\pi : Y \to \P^{2}$ with respect to $\cO_{X}(1)$, which is finite and surjective. As $L$ is $\mu^{h}$-semistable, by \cite[Lemma 2.12]{Pavel2022Thesis} we have 
\begin{equation}\label{eq:bogomolov}
\Delta(\pi_{*}L_{E}) \ge -d^{6} - \beta_d,
\end{equation} 
where $d = \deg_{h^{2}}(Y)$,
\[
    \beta_d = \begin{cases} 
          0 & \text{if }\car(k) = 0, \\
          \frac{d^2(d-1)^2}{(p-1)^2} & \text{if }\car(k) = p > 0 
       \end{cases}
\]
and
\begin{equation}\label{eq:delta}
\Delta := \ch_{1}^{2} - 2d\ch_{2}.
\end{equation} 

The Hilbert polynomial of $L$ with respect to $\cO_{X}(1)$ coincides with that of $\pi_{*}L$ with respect to $\cO_{\P^{2}}(1)$. 
Hence the last coefficient $a_{0}$ of $P_{L}$ equals 
\begin{align*}\label{eq:Hilbcoeff}
a_{0} ={}& \ch_{0}(\pi_{*}L)\Todd_{2}(\P^{2}) + \ch_{1}(\pi_{*}L)\Todd_{1}(\P^{2}) + \ch_{2}(\pi_{*}L)\Todd_{0}(\P^{2}) \\ ={}& d\xi^{2} + \frac{3}{2}\ch_{1}(\pi_{*}L)\xi + \ch_{2}(\pi_{*}L),
\end{align*}
where $\xi = [\cO_{\P^{2}}(1)]$. Replacing \eqref{eq:bogomolov} and \eqref{eq:delta} in the above expression we get $$a_{0} \leq d\xi^{2} + \frac{3}{2}\ch_{1}(\pi_{*}L)\xi + \frac{1}{2d}(\ch_{1}(\pi_{*}L)^{2}+d^{6} + \beta_d).$$
Note that the second coefficient of $P_{L}$ equals $$\ch_{1}(\pi_{*}L)\xi + \frac{3}{2}d \xi^{2},$$ and therefore $\ch_{1}(\pi_{*}L)\xi$ and $\ch_{1}(\pi_{*}L)^{2}$ are bounded. 
This shows that $a_{0}$ is upper bounded as claimed.
\end{proof}

\begin{notation}
Fix  $a\in\R$, 
$K\subset \Amp^2(X)\times\Amp^1(X)$
a compact subset, $r$ and $b$ rational numbers and a class $\beta \in N^{2n-2}(X)$. 
Let $$\Sigma = \Sigma(r,\beta,b,a,K)$$ denote the set of isomorphism classes of $(a,K)$-semistable two-dimensional coherent sheaves $E$ on $X$ of multiplicity $r$, with $\theta_{1}(E)=\beta$ and $\chi(E) \geq b$. 
\end{notation}

The rest of this Section is devoted to the proof of the following.

\begin{thm}\label{thm:boundedness2Dim}
For $r,\beta,b,a,K$ fixed as above, the set $\Sigma(r,\beta,b,a,K)$ is bounded.
\end{thm}

As mentioned in Remark \ref{rem:compacta} it is enough to prove the above boundedness statement for compact sets of the form  $K:=\{h^2\}\times K_1$. We will therefore work with  compacta of this form for the rest of this Section.

We start with some preparations.   Note first that the condition for a pure sheaf $E$ to be $(a,K)$-semistable is equivalent to asking for the existence of some $\alpha \in K_1$ such that for any coherent subsheaf $F\subset E$ with $0 < \mult(F) < \mult(E)$ we have 
\begin{equation}\label{eq:subsheaf}
\mu^{\alpha}(F) \le \mu^{\alpha}(E) + \frac{\mult(E)-\mult(F)}{\mult(F)}a.
\end{equation}

{For any two-dimensional pure sheaf $E$ on $X$, we denote by
\[
    Y(E) = \sum \nu_{Y,E} [Y]
\]
the two-dimensional support cycle of $E$, where the sum is over the irreducible components of $\Supp(E)$. Here $\Supp(E)$ denotes the support of $E$, endowed with the reduced scheme structure. If $Z \subset X$ is a pure two-dimensional subscheme, we simply write $[Z] := [\cO_Z]$.

We will use the definition of the {\em dual sheaf } $E^D$ of a coherent sheaf of dimension $d$ on $X$ following \cite[Definition 1.1.7]{HL}, which is 
$E^D:=\cE xt_X^{n-d}(E,\omega_X).$ The sheaf $E$ will be called reflexive if the natural morphism $E\to E^{DD}$
is an isomorphism. We will also say that a sheaf $E$ has property $(S_2)$ if
\[
  (S_2): \  \depth(E_x) \ge \min\{2,\codim(x,\Supp(E)\} \text{ for all }x \in \Supp(E).
\]
We recall also the $(S_{2,c})$ property (discussed in \cite[Section 1]{HL}) for coherent sheaves $E$ of codimension $c:=n-\dim(E)$ on $X$, which is 
\[
 (S_{2,c}): \   \depth(E_x) \ge \min\{2,\dim(\cO_{X,x})-c\} \text{ for all }x \in \Supp(E).
\]
Note that the $(S_2)$ property is equivalent to the $(S_{2,c})$ property  for pure sheaves of codimension $c$ in $X$. Thus by \cite[Proposition 1.1.10]{HL} a sheaf $E$ of codimension $c$ on $X$ is reflexive if and only if $E$ has $(S_{2,c})$ if and only if $E$ is pure and has $(S_2)$.

\begin{Lem}\label{lem:duality}
Let $E$ be a 2-dimensional pure sheaf on $X$ satisfying $(S_{2})$. Then $$H^{i}(X,E^{D}) \cong H^{2-i}(X,E)^{\vee}$$ for $0 \leq i \leq 2$.
\end{Lem}
\begin{proof}
We have seen  that $E$ is pure and has $(S_{2})$ if and only if $E$ has $(S_{2,c})$ where $c$ is the codimension of the support of $E$ in $X$. By \cite[Proposition 1.1.6]{HL} we see that under the $(S_{2,c})$ assumption on $E$ $$\cE xt^{q}_{X}(E,\omega_{X}) = 0, \quad \text{for }q \neq c.$$

By the local-to-global Ext spectral sequence,
$$H^{p}(X,\cE xt^{q}_{X}(E,\omega_{X})) \implies \Ext^{p+q}_{X}(E,\omega_{X}),$$ we deduce that $$H^{i}(X,E^{D}) = H^{i}(X,\cE xt^{c}_{X}(E,\omega_{X})) \cong \Ext^{i+c}_{X}(E,\omega_{X})$$ for any $i$. By Serre Duality on $X$, the latter space is isomorphic to $H^{2-i}(X,E)^{\vee}$, which proves the claim.
\end{proof}

\begin{cor}\label{cor:E^D}
If $E$ is a two dimensional reflexive sheaf on $X$, one has $\theta_i(E^D)=(-1)^i\theta_i(E)$
for all $i\in\{0,1,2\}$. 
\end{cor}

\begin{proof}
By Lemma \ref{lem:duality} we have for any integer ample class $h\in\Amp^1(X)$ and any integer $m$:
$$P_E^{ h}(m)=\chi(E(m))=\chi(E^D(-m))=P_{E^D}^h(-m)$$
and the equality $\theta_i(E^D)=(-1)^i\theta_i(E)$ follows.
\end{proof}

\begin{Lem}\label{lem:boundedSupp}
    For $r,\beta,b,a,K$ fixed as above, the two-dimensional support cycles $Y(E)$ with $[E] \in \Sigma(r,\beta,b,a,K)$ are bounded.
\end{Lem}
\begin{proof}
For any $E \in \Sigma$, the multiplicity $\mult(E)$ is given by
\begin{align*}
    \mult(E) = \sum \nu_{Y,E} \mult(Y).
\end{align*}
As the multiplicity of each $E \in \Sigma$ is fixed to be $r$, it follows that the multiplicities $\nu_{Y,E}$ and each $\mult(Y)$ in the sum is bounded. 

Let $\mathfrak{Y}$ be the set of all prime cycles in the two-dimensional support cycle of $E$ with $[E] \in \Sigma$. As the multiplicity function (with respect to $h^2$) is bounded on $\mathfrak{Y}$, it follows by Chow's Lemma (see \cite[Lemme 2.4]{Grothendieck-theHilbertScheme}) that $\mathfrak{Y}$ is a bounded family.
\end{proof}
}

\begin{proof}[Proof of Theorem \ref{thm:boundedness2Dim}]\  

\noindent \textbf{Step 1.} Reduction to the $(S_{2})$ case.

If $E^{DD}$ is the reflexive hull of $E$, then there is a short exact sequence
\begin{equation}\label{eq:reflexiveDD}
0 \to E \to E^{DD} \to T \to 0,
\end{equation} with $T$ supported in codimension $\ge 2$ in $\Supp(E)$. This gives $$\chi(E^{DD}) \ge \chi(E) \ge b.$$ Therefore $[E^{DD}] \in \Sigma(r,\beta,b,a,K)$ as soon as $[E] \in \Sigma$, and if we bound the classes of $E^{DD}$, this will bound $\chi(E)$ and $\chi(T)$. Thus the zero-dimensional sheaves $T$ will be bounded, and from \eqref{eq:reflexiveDD} the classes $[E] \in \Sigma$ will also be bounded.

Hence in what follows we may assume that each $[E] \in \Sigma$ is an $(S_2)$-sheaf.  We argue now by induction on $r=\mult(E)$. \medskip

\noindent\textbf{Step 2.} Base case $r=1$. \medskip

Since $r=1$, for any $[E] \in \Sigma(r,\beta,b,a,K)$ there is no subsheaf $F \subset E$ with $\mult(F) < \mult(E)$.  Hence the pure dimensional coherent sheaves $E$ with $[E] \in \Sigma(r,\beta,b,a,K)$ are slope semistable with respect to any ample polarization on $X$ and in particular with respect to $h$. Thus their isomorphism classes form a bounded set by \cite[Theorem 1.8]{Maruyama-Construction}. \medskip

\noindent\textbf{Step 3.} Induction step.\medskip

\noindent\textbf{Claim 1.} There is some $m_{0} = m_{0}(r,\beta,b,a,K)$ such that for every  $(S_2)$-sheaf $[E] \in \Sigma$ we have $$H^{0}(X,E(m)) \ne 0$$ for $m \geq m_{0}$.\medskip

To prove this, first we check that there is some $m' = m'(r,\beta,b,a,K)$ such that for every $[E] \in \Sigma$ we have $H^{2}(X,E(m)) = 0$ for $m \geq m'$. Let $Y$ be the scheme theoretic support of $E$.
By Lemma \ref{lem:duality} we have $$H^{2}(X,E(m)) \cong H^{0}(X,E^{D}(-m))^{\vee}.$$ Suppose there is a non-zero section in $H^{0}(X,E^{D}(-m))$. This section seen as a morphism $\cO_X\to E^{D}(-m)$ factors through a closed subscheme $Z$ of $Y$ and induces injective morphisms $\cO_{Z} \to E^{D}(-m)$ 
and $\cO_{Z}(m) \to E^{D}$. Denote by $F$ the image of the dual morphism $E\to (\cO_Z)^D(-m)$. We have on one hand 
$$\mu^{\alpha}(E) - a\le\mu^{\alpha}(F)$$
for some $\alpha\in K_1$, since $E$ is $(a,K)$-semistable. 
On the other hand from the inclusions $F\subset (\cO_Z)^D(-m)$, and $\cO_Z\subset(\cO_Z)^{DD}$,  we get
\begin{align*}
\mu^{\alpha}(F)\le {}& \mu^\alpha((\cO_Z)^D(-m))= -\mu^\alpha((\cO_Z)^{DD}(m)) 
\le 
-\mu^\alpha(\cO_Z(m))\\ = {}& -\mu^{\alpha}(\cO_Z) - m \frac{\alpha h \theta_2(\cO_Z)}{r'},
\end{align*}
where $r'$ is the multiplicity of $Z$.
Thus 
$$m \frac{\alpha h \theta_2(\cO_Z)}{r'}\le -\mu^{\alpha}(E) +a-\mu^{\alpha}(\cO_Z).$$
By Lemma \ref{lem:boundedSupp}
we know that the support cycle of $E$ is bounded, hence also the cycle $[Z]$ is bounded. Since $Z$ is moreover pure, it follows using Grothendieck's boundedness result \cite[Lemme 2.5]{Grothendieck-theHilbertScheme} that $\deg_{\tilde h}(\cO_Z)$ is bounded from below for any integer ample class $\tilde h\in\Amp^1(X)$. Any $\alpha \in K_1$ may be expressed  as a positive combination of finitely many fixed integer ample $1$-classes. We thus have a uniform lower bound, say $-C$, for $\deg_{\alpha}(\cO_Z)$, when $\alpha$ runs in $K$. Hence
$$m \frac{\alpha h \theta_2(\cO_Z)}{r'}\le -\mu^{\alpha}(E) +a+\frac{C}{r'},$$
 which leads to the existence of $m'$ as desired. 

Thus for $m \ge m'$, we have 
\begin{align*}
h^{0}(X,E(m)) ={}& \chi(E(m)) + h^{1}(X,E(m)) - h^{2}(X,E(m)) \\
\ge{}& \chi(E(m)) \geq r\frac{m^{2}}{2}+ m \beta h + b,
\end{align*}
hence the claim.\medskip

\textbf{Claim 2.} There is some $m_{1} = m_{1}(r,\beta,b,a,K)$ such that for every  $(S_2)$-sheaf $[E] \in \Sigma$ there is some irreducible component $Y_1$ of $\Supp(E)$ and a non-trivial morphism $\cO_{Y_1} \to E(m_1)$.\medskip

To prove this, let $\sigma\in H^{0}(X,E(m_0))$ be some non-zero section, which exists by Claim 1. Let $Y$ be the scheme theoretic support of $E$ and let $[Y]=\sum_{i=1}^p\nu_i[Y_i]$ be its associated two-dimensional cycle (with positive coefficients). 
 The section $\sigma$ seen as a morphism $\cO_X\to E(m_0)$ factors through a closed subscheme $Z$ of $Y$ and induces an injective morphism $\cO_{Z} \to E(m_0)$. The scheme $Z$ is pure two-dimensional, since $E$ is pure. We suppose that $Y_1$ is an irreducible component of the associated cycle $[Z]$ of $Z$.
 
We first check that the scheme $Z$ is bounded. 
We have seen that the support cycle $[Z]$ of $Z$ is bounded when $[E]$ belongs to $\Sigma$, by Lemma \ref{lem:boundedSupp}. Thus the function $\deg_{\tilde h^2}(\cO_Z)$ is bounded for any rational ample class $\tilde h\in\Amp^1(X)$. 
By Grothendieck's Lemma it follows that the function $\deg_{\tilde h}(\cO_Z)$ is lower bounded when $[E]$ belongs to $\Sigma$ and that an upper bound for such a function would imply boundedness for $(\cO_Z)_{(2)}=\cO_Z$. This will be achieved if we show that $\theta_1(\cO_Z)$ is bounded. 

The existence of an  injective morphism $\cO_{Z} \to E(m_0)$ combined with  the semistability condition on $E$  gives an upper bound for $\deg_{\alpha}(\cO_Z)$ as follows.
Let $r'$ be the multiplicity of $Z$ (with respect to $h$). By equation \eqref{eq:subsheaf} have 
$$\mu^{\alpha}(\cO_Z)= \mu^{\alpha}(\cO_Z(-m_0)) + m_0 \frac{\alpha h \theta_2(\cO_Z)}{r'} \le \mu^{\alpha}(E)+ m_0 \frac{\alpha h \theta_2(\cO_Z)}{r'}+\frac{(r-r')a}{r'},$$ 
hence
$$\deg_{\alpha}(\cO_Z) \le 
\frac{r'}{r}\deg_{\alpha}(E)+ m_0 \alpha h \theta_2(\cO_Z)+(r-r')a,$$ which gives 
an upper bound for $\deg_{\alpha}(\cO_Z)$. 
After possibly cutting $K$ into finitely many smaller pieces, we may assume that there exist
 ample classes $h_1,\ldots,h_\rho\in \Amp^1(X)$ spanning $N^1(X)$, such that $K$ lies in the interior of the convex cone generated by $h_1,\ldots,h_\rho$. In particular we may write $\alpha $ as a combination $\sum_{j=1}^\rho a_jh_j$ with positive coefficients $a_j>0$. 
 Using the lower bounds on the degrees of $\theta_1(\cO_Z)$ with respect to 
 the classes $h_j$ as well as the upper bound with respect to $\alpha$ and Lemma \ref{lem:convexity} we get the desired boundedness for $\theta_{1}(\cO_Z)$ and consequently for $\cO_Z$.

Let now $Y_1$ be any irreducible component of $Z$. Then the sheaf $\cH om_X(\cO_{Y_1}, \cO_Z)$ is non-trivial and is bounded since $\cO_{Y_1}$ and $ \cO_Z$ are bounded. There exists therefore some integer $m''$ depending only on $(r,\beta,b,a,K)$ such that $$H^0(X, \cH om(\cO_{Y_1}, \cO_Z)(m''))\ne0.$$ We thus get a  non-trivial morphism $\cO_{Y_1}\to \cO_Z(m'')$, whence the desired non-trivial morphism  $\cO_{Y_1}\to E(m_0+m'')$. Claim 2 is proved. \medskip

\textbf{Claim 3.} Putting $m=m_1$ as in Claim 2, one can associate to each $E$ with $[E]\in\Sigma$  a saturated subsheaf $L_{E}$ of $ E(m)$ such that the scheme theoretic support $\SchSupp(L_E)$ of $L_E$ is integral, $L_E$ has rank one on its support, $\theta_{1}(L_{E})$ is bounded  and  $\chi(L_{E})$  is upper bounded in terms of $(r,\beta,b,a,K)$.\medskip

For the proof we take up the notations of  Claim 2, and consider the saturation $L_E$ of $\cO_{Y_1}$ in $E(m)$, where $m=m_1$, and let $E_{1} := E(m)/L_{E}$.  Then the multiplicity of  $L_{E}$ (with respect to $\cO_{X}(1)$) 
is bounded since $L_{E}$ is a subsheaf of $E(m)$ and the mutiplicity of  $E(m)$ is bounded. 

For the boundedness of $\theta_{1}(L_{E})$, we will use Lemma \ref{lem:convexity}. For this we have to bound degrees of $\theta_{1}(L_{E})$ as we have done for the degrees of $\theta_{1}(\cO_Z)$
in the proof of Claim 2. 
Note first that we easily get a lower bound for the degree of $\theta_{1}(L_{E})$ with respect to any ample class $\tilde h\in\Amp^1(X)$, since   $\cO_{Y_1}$ injects into $L_{E}$ with at most $1$-dimensional quotient and therefore $\deg_{\tilde h}(L_E)\ge\deg_{\tilde h}\theta_1(\cO_{Y_1})$,  the latter degree being bounded by Lemma \ref{lem:boundedSupp}.  We next obtain an upper bound on $\deg_\alpha(L_{E})$ for some $\alpha\in K$ using the fact that $E$ is $(a,K)$-semistable exactly 
as
in the proof of Claim 2 and the desired boundedness for $\theta_{1}(L_{E})$ follows as before. 

Finally $\chi(L_E)$ is upper bounded by Lemma \ref{lem:rank1inequality} and Claim 3 is proved.\medskip

\textbf{Claim 4.} The set of isomorphism classes of sheaves $E_{1} = E(m)/L_{E}$ with $[E] \in \Sigma$ is bounded in terms of $(r,\beta,b,a,K)$.\medskip

For the proof we first note that $E_1$ is either zero or purely $2$-dimensional, since $L_E$ is saturated in $E(m)$. The case $E_1=0$ being clear, we place ourselves in the case when $E_1$ is pure of dimension $2$.

Next note that the multiplicity $r_{1}$ of $E_{1}$ equals $r - \deg_{h^{2}}(L_{E})$, hence is bounded. Similarly we get boundedness for $\theta_{1}(E_{1})$ which equals  $$\theta_{1}(E(m))-\theta_{1}(L_{E}) = \theta_{1}(E)+mh\theta_2(E) - \theta_{1}(L_{E}).$$ 
As the last coefficient of the Hilbert coefficient of $L_{E}$ (with respect to $h$) is upper bounded, and that of $E(m)$ is lower bounded, we get a lower bound $b_{1}$ for the last coefficient of the Hilbert polynomial of $E_{1}$. 

Let $F$ be a pure quotient of $E_{1}$. This is also a quotient of $E(m)$, so since $E$ is $(a,K)$-semistable, there exists some $\alpha \in K_1$ such that $$\mu^{\alpha}(F(-m)) \ge \mu^{\alpha}(E) - a.$$ Using $\deg_{\alpha}(F(-m)) = \deg_{\alpha}(F) - m\alpha h \theta_2(F)$ and the additivity of the degree function $\deg_\alpha$ in the short exact sequence
\[
    0 \to L_E(-m) \to E \to E_1(-m) \to 0,
\]
we obtain

\begin{align*}
\mu^{\alpha}(F) \ge m \frac{\alpha h \theta_2(F)}{\mult(F)} + \frac{1}{r}(\deg_{\alpha}(E_{1}(-m)) + \deg_{\alpha}(L_{E}(-m))-a.
\end{align*}
Using the identity $\deg_{\alpha}(E_{1}(-m)) = \deg_{\alpha}(E_{1}) - m\alpha h \theta_2(E_1)$, we get by a straightforward computation that $\mu^{\alpha}(F)$ is greater or equal than
\[
 \mu^{\alpha}(E_{1}) - \left[\frac{r-r_{1}}{r}\mu^{\alpha}(E_{1}) + m\frac{\alpha h \theta_2(E_{1})}{r} - m\frac{\alpha h \theta_2(F)}{\mult(F)} - \frac{\deg_{\alpha}(L_{E}(-m))}{r} + a\right].
\]

Since the last bracket is bounded, we now get that $E_{1}$ is $(a_{1},K)$-semistable for a suitable $a_{1}$, hence $[E_{1}]$ is in a finite union of sets of type $\Sigma(r_{1},\beta_{1},b_{1},a_{1},K)$ and hence bounded by the induction hypothesis.\medskip

\textbf{Claim 5.} The set $\Sigma$ is bounded.\medskip

Keeping our previous notations we see that by Claim 4 we obtain also a lower bound on $\chi(L_E)$, hence boundedness for the sheaves $L_E$ by Step 2. Together with the boundedness for $E_1$ this implies the boundedness of $\Sigma$ and we are done.
\end{proof}

\begin{cor}\label{cor:UniformBoundednessD=2}
    Let $\gamma$ be a numerical class of $2$-dimensional coherent sheaves on $X$. Then any compact set $K\subset \Amp^2(X)\times\Amp^1(X)$ is a bounded stability parameter set for $\gamma$. 
\end{cor}

{As a consequence we also obtain a Bogomolov type inequality as in \cite[Corollary 2.10]{Maruyama81boundedness}. For the statement we need some notation. Let $E$ be a pure 2-dimensional sheaf on $X$ and denote by $Y_E$ its scheme-theoretic support. Consider any (finite, surjective) linear projection $\pi : Y_E \to \P^2$ with respect to $\cO_X(1)$, and set
\[
    \Delta_h(E) := 2\rk(\pi_*E)c_2(\pi_*E) - (\rk(\pi_*E)-1)c_1(\pi_*E)^2.
\]
This is well defined since the Chern character of $\pi_*E$ depends only on (the numerical type of) $E$. In fact one can also write $\Delta_h(E)$ in terms of the classes $\theta_j(E)$, $0 \le j \le 2$.

\begin{cor}\label{cor:BogomolovIneq} 
Let $a \in \R$ and let $K \subset \Amp^1(X)$ be a compact. There is a constant $C(r,a,K)$ depending only on $r, a$ and $K$ such that 
\[
    \Delta_h(E) \ge C(r,a,K)
\]
for all 2-dimensional $(a,K)$-semistable sheaves $E$ of fixed multiplicity $r$ on $X$.
\end{cor}
\begin{proof}
Let $\cF'$ be the family of 2-dimensional $(a,K)$-semistable sheaves $E$ of fixed multiplicity $r$ on $X$, and let $\cF = \{ (\pi : Y_E \to \P^2)_*E \mid E \in \cF' \}$ be the induced family of sheaves on $\P^2$. Then the result follows by applying the argument in \cite[Corollary 2.10]{Maruyama81boundedness} to the family $\cF$ and using Theorem \ref{thm:boundedness2Dim}. We omit the details.
\end{proof}
\begin{Rem}
If $a = 0$ and $K = \{ h \}$, then one can take $C(r,a,K) = -r^6-\beta_r$, cf. \cite[Lemma 2.12]{Pavel2022Thesis}, where $\beta_r$ is as defined in Lemma \ref{lem:rank1inequality}.
\end{Rem}
}

\section{A generalized Grauert-Mülich-Spindler Theorem}\label{section:GrauertMuelich}

{From now on we assume that $X$ is defined over an algebraically closed field of characteristic zero.} {In this section we prove an analogue of the Grauert-Mülich-Spindler Theorem for $(a,\alpha h^2,\alpha h)$-semistable  $d$-dimensional sheaves on $X$, where $\alpha \in \Amp^{d-2}(X)$ and $a \in \R$ are fixed. This is our main ingredient in showing uniform boundedness results in higher dimensions.}

Set $\Pi = |\mathcal O_X(1)|$ and consider the 
 incidence variety $Z = \{ ([D],x) \in \Pi \times X  \mid x \in D \}$ with its natural projections:
\begin{center}
\begin{tikzcd}[column sep=normal]
Z \ar[d,"p"] \ar[r,"q"] &  X \\  \Pi
\end{tikzcd}
\end{center}

One can show that $q: Z \to X$ is a projective bundle. Also, if $[D] \in \Pi$ is a closed point, then the fibre $p^{-1}([D])$ is identified by $q$ with $D \subset X$. 

\subsection*{Relative HN filtration.} Let $E$ be a purely $d$-dimensional sheaf on $X$. By the analogue \cite[Theorem 3.2]{MegyPavelToma} of \cite[Thm.~2.3.2]{HL} for our notion of $\mu^\alpha_\beta$-semistability, there exists a relative Harder-Narasimhan filtration of $\cF = q^*E$
\begin{align*}
    0 = \cF_0 \subset \cF_1 \subset \ldots \subset \cF_l = \cF
\end{align*}
and an open subset $U \subset \Pi$ such that

1. $\cF_i/\cF_{i-1}$ is flat over $U$ for each $i=1,\ldots,l$,

2. $0 = \cF_0|_{D} \subset \cF_1|_{D}  \subset \ldots \subset \cF_l|_{D} = E|_{D}$ is the Harder-Narasimhan filtration of $E|_{D}$ for each $D \in U$, with $\mu_1 > \ldots > \mu_l$, where $\mu_i = \mu^\alpha_{\alpha h}((\cF_i/\cF_{i-1})|_{D})$. As $U$ is connected and the $\cF_i/\cF_{i-1}$ are flat over $U$, the $\mu_i$ do not depend on $D \in U$. Moreover, by Bertini's Theorem, we may assume that each $D \in U$ is an intregral smooth divisor on $X$.

\begin{thm}\label{thm:GMSthmHD}
    Under the above notation, assume further that $E$ is $d$-dimensional and $(a,\alpha h^2,\alpha h)$-semistable on $X$. Then, for general $D \in |\mathcal O_X(1)|$ and $i = 1,\ldots,l-1$, one of the following holds:
    \begin{enumerate}
        \item $\mu^\alpha_{\alpha h}((q^*E/\cF_i)|_D) \ge \mu^\alpha_{\alpha h}(E) - \frac{1}{2} - a$, or
        \item $\mu_i - \mu_{i+1} \le 1$.
    \end{enumerate}
\end{thm}
\begin{proof}
Suppose that $\mu^\alpha_{\alpha h}((q^*E/\cF_i)|_D) < \mu^\alpha_{\alpha h}(E) - \frac{1}{2} - a$ for some $i$, and set $F' = \cF_i$, $F'' = \cF/F'$. See \cite[Prop.~1.11]{flenner1984restrictions} for the following result.
\begin{prop}
Let $Z$, $Y$ be $\Q$-schemes and $q: Z \to Y$ a surjective smooth mapping of finite presentation with geometrically connected fibers. Let $\cT_{Z/Y}$ be the relative tangent sheaf. If $E$ is a coherent sheaf on $Y$ and $F'\subseteq q^*E$ is a coherent subsheaf such that
\begin{align*}
    \Hom_{\cO_Z}(\cT_{Z/Y} \otimes F',F'') = 0, \quad F'' = q^*E/F',
\end{align*}
then there exists a unique coherent subsheaf $F$ of $E$ with $F' = q^*F$.
\end{prop}

We apply this result to the morphism $q : Z \to X$. Suppose for contradiction that $\Hom(\cT_{Z/X} \otimes F',F'') = 0$. Then there exists a subsheaf $F \subseteq E$ such that $F' = q^*F$, hence there is also a quotient $E \to Q$ such that $F'' = q^*Q$. As $Q|_D = F''|_D$, we have
\[
    \mu^\alpha_{\alpha h}(Q|_D) < \mu^\alpha_{\alpha h}(E) - \frac{1}{2} - a,
\]
which gives
\[
    \mu^\alpha_{\alpha h}(Q) < \mu^\alpha_{\alpha h}(E) - a,
\]
by using Lemma \ref{lemma:destabilizingRestriction} below. This contradicts the fact that $E$ is $(a,\alpha h^2,\alpha h)$-semistable.

\begin{Lem}\label{lemma:destabilizingRestriction}
We have $\mu^\alpha_{\alpha h}(F|_D) = \mu^{\alpha h}_{\alpha h^2}(F) - 1/2$ for general $D \in |\mathcal O_X(1)|$.
\end{Lem}
\begin{proof}
This uses the additivity of degree functions in short exact sequences. By using the short exact sequence
\begin{align*}
   0 \to F(-1) \to F \to F|_D \to 0,
\end{align*}
with $D \in |\mathcal O_X(1)|$ general, we obtain
\begin{align*}
    \deg_{\alpha h}(F|_D) &= \deg_{\alpha h^2}(F)  \\
     \deg_\alpha(F|_D)    &= \deg_{\alpha h}(F) - \frac{1}{2}\deg_{\alpha h^2}(F).
\end{align*}
Hence
\begin{align*}
    \mu^\alpha_{\alpha h}(F|_D) = \frac{\deg_{\alpha}(F|_D) }{\deg_{\alpha h}(F|_D)} = \mu^{\alpha h}_{\alpha h^2}(F) - \frac{1}{2}. 
\end{align*}
\end{proof}

Thus we can find a nonzero map $\cT_{Z/X} \otimes F' \to F''$, which restricted to a general fiber of $p$ gives a nonzero map
\begin{align*}
    \cT_{Z/X}|_{D} \otimes F'|_{D} \to F''|_{D}.
\end{align*}
Hence
\begin{equation}\label{eq: slopeIneqI}
    \mu^\alpha_{\alpha h,\min}(\cT_{Z/X}|_{D} \otimes F'|_{D}) \leq \mu^\alpha_{\alpha h,\max}(F''|_{D}).
\end{equation}
From here the proof follows exactly as in \cite[p. 66]{HL}.  Let $\cK$ be the kernel of the evaluation map
\[
    \ev: H^0(X,\cO_X(1)) \otimes \cO_X \to \cO_X(1). 
\]
Then $Z = \P(\cK^\vee)$ and we have the Euler sequence
\[
    0 \to \cO_{Z} \to q^* \cK \otimes p^*\cO(1) \to \cT_{Z/X} \to 0. 
\]
From this and the Koszul complex corresponding to the evaluation map $\ev$ above, we obtain a surjection
\[
 \left(\Lambda^2 H^0(X,\cO_X(1)) \otimes \cO_X(-1)\right)|_D \to \cT_{Z/X}|_{D}.
\]
Hence
\begin{align*}
     \mu^\alpha_{\alpha h,\min}(\cT_{Z/X}|_{D} \otimes F'|_{D}) \ge{}&  \mu^\alpha_{\alpha h,\min}(\Lambda^2 H^0(X,\cO_X(a)) \otimes \cO_X(-1) \otimes F'|_D)\\
    ={}& \mu^\alpha_{\alpha h,\min}(\cO_X(-1) \otimes F'|_D) \\
    ={}& \mu^\alpha_{\alpha h,\min}(F'|_D) - 1.
\end{align*}
Putting this inequality together with \eqref{eq: slopeIneqI} yields
\[
    \mu_i - \mu_{i+1} = \mu^\alpha_{\alpha h,\min}(F'|_D) - \mu^\alpha_{\alpha h,\max}(F''|_D) \le 1,
\]
since $\mu_i = \mu^\alpha_{\alpha h,\min}(F'|_D)$ and $\mu_{i+1} = \mu^\alpha_{\alpha h,\max}(F''|_D)$ by construction.
\end{proof}

\begin{cor}\label{cor:restrictionOfType}
Let $E$ be a $d$-dimensional $(a,\alpha h^2,\alpha h)$-semistable sheaf. Then there is some $a' \in \R$ depending only on $a$ and $\deg_{h^d}(E)$ such that $E|_D$ is  $(a',\alpha h,\alpha)$-semistable for general $D \in |\mathcal O_X(1)|$.
\end{cor}
\begin{proof}
We choose a general divisor $D \in |\mathcal O_X(1)|$ such that $D$ avoids the associated points of $E$ and $E|_D$ remains pure. Also denote by 
\[
    0 = E_0 \subset E_1 \subset \ldots \subset E_l = E|_D
\]
the Harder-Narasimhan filtration of $E|_D$. Now consider a pure quotient $E|_D \to F$, and note that by construction $\mu^\alpha_{\alpha h}(F) \ge \mu^\alpha_{\alpha h,\min}(E|_D)$.
By Theorem \ref{thm:GMSthmHD}, for general $D \in |\mathcal O_X(1)|$ and $i = 1,\ldots,l-1$ we have
\begin{enumerate}
    \item $\mu^\alpha_{\alpha h}(E|_D/E_i) \ge \mu^\alpha_{\alpha h}(E|_D) - a$, or
    \item $\mu^\alpha_{\alpha h}(E_{i+1}/E_{i}) \ge \mu^\alpha_{\alpha h}(E_{i}/E_{i-1})  - 1$.
\end{enumerate}
Using the fact that $\mu^\alpha_{\alpha h}(E_{i+1}/E_{i}) \ge \mu^\alpha_{\alpha h}(E|_D/E_{i})$, we deduce that for each $i$
\begin{enumerate}
    \item[(1')] $\mu^\alpha_{\alpha h}(E_{i+1}/E_{i}) \ge \mu^\alpha_{\alpha h}(E|_D) - a$, or
    \item[(2)] $\mu^\alpha_{\alpha h}(E_{i+1}/E_{i}) \ge \mu^\alpha_{\alpha h}(E_{i}/E_{i-1})  - 1$.
\end{enumerate}
Applying $(1')$ and $(2)$ for $i$ starting with $i = l-1$, we find successively $\mu^\alpha_{\alpha h,\min}(E|_D) = \mu^\alpha_{\alpha h}(E_{l}/E_{l-1}) \ge \mu^\alpha_{\alpha h}(E|_D) - a$ and in this case we stop, or we find
$\mu^\alpha_{\alpha h}(E_{l}/E_{l-1}) \ge \mu^\alpha_{\alpha h}(E_{l-1}/E_{l-2}) -1$ and in this case we make $i = l-2$ and continue. Eventually we get
\[
    \mu^\alpha_{\alpha h}(F) \ge \mu^\alpha_{\alpha h,\min}(E|_D) \ge \mu^\alpha_{\alpha h}(E|_D) - \max(l-1,l-1+a).
\]
We have also used that $\mu^\alpha_{\alpha h}(E_{1}/E_{0}) = \mu^\alpha_{\alpha h,\max}(E|_D) \ge \mu^\alpha_{\alpha h}(E|_D)$.
\end{proof}

\section{Boundedness results in higher dimension}

{We are now ready to prove our main boundedness result for higher dimensional pure sheaves.}
{
\begin{thm}[Uniform boundedness of pure sheaves]\label{thm:boundednessHD}
 Let $d \ge 2$, let $K_1 \subset \Amp^1(X)$ and $K_d \subset \Amp^d(X)$ be compact sets, and let $K_{d-1} = K_1 h_1 \cdots h_{d-2} \subset \Amp^{d-1}(X)$ for some fixed $h_1,\ldots,h_{d-2} \in \Amp^1(X)$ rational ample divisor classes. Fix $a \in \R$, set $K = K_{d} \times K_{d-1}$, and let $\gE$ be the set of isomorphism classes of $d$-dimensional sheaves of fixed numerical type on $X$ which are $(a,K)$-semistable. Then $\gE$ is bounded.
\end{thm}

\begin{cor}\label{cor:higherDimension}
 Let $d \ge 2$ and let $h_1,\ldots,h_{d-2} \in \Amp^1(X)$ be rational ample divisor classes.  Suppose also $K_1 \subset \Amp^1(X)$ and $K_d \subset \Amp^d(X)$ are compact sets, and define 
 $$K_{d-1} = K_1 h_1 \cdots h_{d-2} \subset \Amp^{d-1}(X).$$

 Then the set of isomorphism classes of $d$-dimensional coherent sheaves of fixed numerical type and  slope-semistable with respect to some $(\beta,\alpha)\in K_{d}\times K_{d-1}$ is bounded.
\end{cor}

\begin{proof}[Proof of Theorem \ref{thm:boundednessHD}]
We assume that all sheaves in $\gE$ have fixed classes $\theta_j$, $0 \le j \le d$. We shall argue by induction on $d$. If $d = 2$, then for any $[E] \in \gE$, we have that $E$ is  $(a,\beta,\alpha)$-semistable for some $\beta \in K_2$ and some $\alpha \in K_1$. We can therefore apply Theorem \ref{thm:boundedness2Dim} and obtain that $\gE$ is bounded when $d = 2$.  

Now let $d > 2$. For the induction step, we may assume that the class $h_1$ corresponds to a very ample divisor $H_1$ on $X$, by replacing $h_1$ by some multiple $m h_1$ if necessary. Let $[E] \in \gE$. Then $E$ is  $(a,\beta,\alpha h_1 \cdots h_{d-2})$-semistable for some $\beta \in K_d$ and some $\alpha \in K_1$. According to Lemma \ref{lem:changingDegs} there is some $a' \in \R$, depending on $K, \theta_d, \theta_{d-1}$ and $a$, such that $E$ is  $(a',\alpha h_1^2 h_2 \cdots   h_{d-2},\alpha h_1 \cdots h_{d-2})$-semistable.

By Corollary \ref{cor:restrictionOfType} we obtain that $E|_{D_1}$ is  $(a'',\alpha h_1 \cdots h_{d-2},\alpha h_2\cdots h_{d-2})$-semistable for some general $D_1 \in |H_1|$ and $a''$ depending only on $K, \theta_d, \theta_{d-1}$ and $a$. 
Using again Lemma \ref{lem:changingDegs}, we can find $b \in \R$ (depending on $K, \theta_d, \theta_{d-1}$ and $a$) such that $E|_{D_1}$ is  $(b,h_1^{d-1},\alpha h_2\cdots h_{d-2})$-semistable. Also note that the numerical classes of the restricted sheaves $E|_{D_1}$, with $[E] \in \gE$, depend only on the $\theta_j$, $1 \le j \le d$, and $h_1$. 
Therefore all restricted sheaves $E|_{D_1}$, with $[E] \in \gE$, fit inside a suitable family $\Sigma$ of $(d-1)$-dimensional pure sheaves of fixed numerical type on $X$ which are  $(b,K')$-semistable, where $K' = \{ h_1^{d-1}\}\times  (K_1 h_2\cdots h_{d-2})$. By induction we have that $\Sigma$ is bounded, and so $\mu^{h_1^{d-2}}_{h_1^{d-1},\max}$ is bounded on $\Sigma$. As
\[
    \mu^{h_1^{d-2}}_{h_1^{d-1}}(E|_{D_1}) = \mu^{h_1^{d-1}}_{h_1^{d}}(E) - \frac{1}{2}
\]
for $[E] \in \gE$ and general $D_1 \in |H_1|$, we further get that $\mu^{h_1^{d-1}}_{h_1^{d},\max}$ is upper bounded on $\gE$. It follows that $\gE$ is also bounded, cf. \cite{Langer,HL}.
\end{proof}
}

\begin{rem}
In the proof of Theorem \ref{thm:boundednessHD} our Theorem \ref{thm:GMSthmHD} of Grauert-M\"ullich type and its Corollary \ref{cor:restrictionOfType} are used for the  induction step. Their statements impose that we start the induction at $d=2$, hence the logical need of a separate proof of a boundedness result in the two-dimensional case.
\end{rem}

\begin{cor}\label{cor:boundednessSegments}
Let $\gamma$ be a numerical class of $d$-dimensional coherent sheaves on $X$, let $\alpha, \alpha' \in \Amp^1(X)$ and let $h_1,\ldots,h_{d-2},h_1',\ldots,h'_{d-2} \in \Amp^1(X)$ be rational ample classes. Then there is a connected compact subset $K_{d-1} \subset \Amp^{d-1}(X)$ containing $\alpha h_1\cdots h_{d-2}$ and $\alpha' h'_1\cdots h'_{d-2}$ such that for any non-empty compact set $K_d \subset \Amp^{d}(X)$, the set $K_{d} \times K_{d-1} $ is a bounded stability parameter set for $\gamma$. 
\end{cor}

\begin{proof}

Fix another rational $h\in \Amp^{1}(X)$ and denote by $\ell \subset \Amp^1(X)$ the segment between $\alpha$ and $h$.  Then the set $\kappa: = \ell h_1\cdots h_{d-2}\subset \Amp^{d-1}(X)$ joins $\alpha h_1\cdots h_{d-2}$ and $h h_1\cdots h_{d-2}$.     And for any non-empty compact set $K_d \subset \Amp^{d}(X)$, the set $ K_d\times\kappa$ is a bounded stability parameter set for $\gamma$ by Theorem \ref{thm:boundednessHD}.

Similarly using the line $\ell'$ joining $\alpha'$ and $h$ and the lines $\ell_i$ joining $h_i$ and $h_i'$ in $\Amp^1(X)$ we construct sets $\kappa'$ and $\kappa_i$ for $i=1,\ldots,d-2$.   Now set $K_{d-1} = \kappa\cup \kappa'\cup \kappa_1\cup \cdots \cup \kappa_{d-2}$ and we are done (see Figure \ref{fig:linesofboundedstabilityparameters} for an illustration in the case $d=4$).
\end{proof}

\begin{rem}
    In the special case when $d=n$ the results of this section remain valid also over a ground field of arbitrary characteristic.
    In this case their proof is based on Langer's boundedness result from \cite{Langer} and the technique used in \cite{GRT1} to prove uniform boundedness. Indeed, to get a proof of Corollary \ref{cor:higherDimension} for instance for $d=n$ in arbitrary characteristic it suffices to use in \cite[Theorem 6.8]{GRT1} a slightly modified definition of the set $C^+(X)$ considered therein, by putting 
    $$K^+_{h_1\ldots h_{n-2}}:=\{\beta\in N^1(X)_\R \ | \ \beta^2h_1\ldots h_{n-2}>0, \ \beta h h_1\ldots h_{n-2}>0\}, $$
    $$ C^+_{h_1\ldots h_{n-2}}:= h_1\ldots h_{n-2} K^+_{h_1\ldots h_{n-2}}, $$
    $$ C^+(X):=\bigcup_{(h_1,\ldots,h_{n-2})\in (\Amp^1(X))^{n-2}}C^+_{h_1\ldots h_{n-2}}. $$
    \end{rem}

\section{Applications}\label{sec:corollaries}

As announced in the introduction, the main application of our boundedness results are to wall-crossing with respect to variation of stability notions. With the exception of Proposition \ref{prop:MPT} and Corollary \ref{cor:wallsD=2}, which are valid in any characteristic, for the results of this section we need to assume that our base field is of characteristic zero. We briefly recall the setup proposed in \cite{MegyPavelToma}.

Fix $d$ an integer with $2\le d\le n$ and let $E$ be a $d$-dimensional coherent sheaf on $X$.
For a fixed element $\alpha = (\alpha_d,\ldots,\alpha_0)\in \Amp^d(X)\times\ldots\times\Amp^0(X)$ the {\em $\alpha$-Hilbert polynomial} and the  \textit{reduced $\alpha$-Hilbert polynomial} of $E$ are defined by $$
	P_\alpha(E,m) = \sum_{i=r}^d \deg_{\alpha_i}(E)\frac{m^i}{i!} \ 
 \text{and  by} \ 
  p_\alpha(E) = \frac{P_\alpha(E)}{\deg_{\alpha_d}(E)}
$$
respectively. The sheaf $E$ is then said to be {\em $\alpha$-semistable} if it is pure  and
	for all pure $d$-dimensional quotients $E\to F$ one has 
	\[
		 p_\alpha(F,m) \geq p_\alpha(E,m) \text{ for }m \gg 0.
	\]
When $\alpha=(h^d,\ldots,h^0)$ the usual notion of semistabilty in the sense of Gieseker-Maruyama-Simpson is recovered by the above definition. Clearly  an $\alpha$-semi-stable sheaf in the above sense is necessarily $(\alpha_d,\alpha_{d-1})$-slope-semistable  in the sense of Definition \ref{def:stability}. 

Fix now $\gamma$ a numerical class of $d$-dimensional coherent sheaves and let $\alpha = (\alpha_d,\ldots,\alpha_0)\in \Amp^d(X)\times\ldots\times\Amp^0(X)$ be as above. Then the substack $\cC oh_X^{\gamma, \alpha ss}$ of $\alpha$-semistable coherent sheaves of numerical type $\gamma$ on $X$ is open in the algebraic stack $\cC oh_X$ of coherent sheaves on $X$, 
\cite[Proposition 4.1]{MegyPavelToma}. Under a suitable boundedness assumption and making use of \cite[Theorem 7.2.7]{AlperHLH} it is further proved in \cite[Theorem 5.1]{MegyPavelToma} that $\cC oh_X^{\gamma, \alpha ss}$ admits a proper good moduli space $M_X^{\gamma, \alpha ss}$ in the sense of \cite{Alper13}. Combining this with our Corollary \ref{cor:higherDimension} we get
\begin{cor}
    If $\gamma$ is a numerical class of $d$-dimensional coherent sheaves on $X$ and $\alpha\in \Amp^d(X)\times\ldots\times\Amp^0(X)$ is such that $\alpha_{d-1}$ is a complete intersection class $\alpha_{d-1}=h_1\cdots h_{d-1}$ with $h_1\in \Amp^1(X)$ real arbitrary and $h_2,\ldots,h_{d-1}\in \Amp^1(X)$ rational, then the stack $\cC oh_X^{\gamma, \alpha ss}$ of $\alpha$-semistable coherent sheaves of numerical type $\gamma$ on $X$ admits a proper good moduli space. 
\end{cor}

Again under suitable boundedness assumptions, the following existence result for locally finite rational chamber structures for the variation of semistability was proved in \cite{MegyPavelToma}. 

\begin{prop}[\cite{MegyPavelToma}]\label{prop:MPT}
    Let $d$ and $\gamma$ be as above and let $K\subset\Amp^d(X)\times\ldots\times\Amp^0(X)$ be a compact set whose projection on $N^d(X)\times N^{d-1}(X)$ is a bounded stability parameter set for $\gamma$.
    Then there is a finite set $W_{\gamma,K}$ of  hypersurfaces in $N^d(X)_\R\times\ldots\times N^0(X)_\R$ given by rational bilinear forms accounting for the variation of $\alpha$-semistability for coherent sheaves in the numerical class $\gamma$  for $\alpha\in K$ in the following sense. If $E$ is a coherent sheaf in $\gamma$ and $\alpha', \alpha''\in K$ lie in the same connected component of the complement of the hypersurfaces from $W_{\gamma,K}$  in $N^d(X)_\R\times\ldots\times N^0(X)_\R$, then $E$ is $\alpha'$-semistable if and only if it is $\alpha''$-semistable.
\end{prop}

Due to the boundedness result of Corollary \ref{cor:boundednessSegments}, we obtain the following.
\begin{cor}
Let $d$ be an integer with $2 \le d \le n$ and let $\gamma$ be a numerical class of $d$-dimensional coherent sheaves on $X$. Let $h_1, h_2 \in \Amp^1(X)$ be two distinct rational ample classes. Then the points $(h_1^d,\ldots,h_1^0)$ and $(h_2^d,\ldots,h_2^0)$ in $N^d(X)_\R \times \ldots \times N^0(X)_\R$ can be connected by a polygonal chain $L\subset \Amp^d(X)\times\ldots\times\Amp^0(X)$ and there are finitely many rational points $\alpha^{(1)},\ldots, \alpha^{(m)}$ on $L$ accounting for the variation of $\alpha$-semistability for coherent sheaves in the numerical class $\gamma$ for $\alpha \in L$. Moreover, $L$ may be chosen such that for each $i$ with $1\le i\le m$ all components $\alpha_j^{(i)}$,  $d\ge j\ge0$, of the point $\alpha^{(i)}$ are complete intersections of rational classes on the segment $[h_1,h_2]\subset  \Amp^1(X)$.
\end{cor}

 In the particular case when  $d=2$, we have most comprehensive uniform boundedness results, Theorem \ref{thm:boundedness2Dim} and Corollary \ref{cor:UniformBoundednessD=2}. In particular, in this case any compact set  $K\subset \Amp^2(X)\times\Amp^1(X)\times \Amp^0(X)$ fulfills the boundedness condition in Proposition \ref{prop:MPT}. Hence we get the following.

 \begin{cor}\label{cor:wallsD=2}
 Let $d=2$ and let $K\subset \Amp^2(X)\times\Amp^1(X)\times \Amp^0(X)$  be a compact subset. Let $\gamma$ be as above. Then there is a finite set $W_{\gamma,K}$ of  hypersurfaces described by rational bilinear forms on  $N^2(X)_\R\times N^1(X)_\R$ accounting for the variation of $\alpha$-semistability for coherent sheaves in the numerical class $\gamma$ for $\alpha\in K$.
\end{cor}

This applies in particular to compact sets of the form $\{(\omega^2,\omega,\omega^0)\in N^2(X)_\R\times N^1(X)_\R\times N^0(X)_\R \ | \ \omega\in K_1\}$, where $K_1\subset\Amp^1(X)$ is compact, as considered for instance in \cite[Theorem 7.24]{joyce2021enumerative}. 

Also for $d=2$ Theorem \ref{thm:boundedness2Dim} further allows one to apply the results of \cite{GRT1}, \cite{GrebRossToma-Semicontinuity}, \cite{GRT3}. In particular, in combination with \cite[Corollary 4.4]{GRT3} it gives 

\begin{cor}\label{cor:MasterSpace}
    Let $\gamma$ be a numerical class of $2$-dimensional coherent sheaves on $X$ and $h_1, h_2 \in \Amp^1(X)$ be two  rational ample classes. Then the moduli spaces $M_X^{\gamma, (h_1^2,h_1,h_1^0) ss}$, $M_X^{\gamma, (h_2^2,h_2,h_2^0) ss}$ of Gieseker-semistable  sheaves in $\gamma$ with respect to $h_1$ and $h_2$, respectively, are related by a finite number of Thaddeus flips.
\end{cor}

\bibliography{bib2021}
\bibliographystyle{amsalpha}

\medskip
\medskip
\begin{center}
\rule{0.4\textwidth}{0.4pt}
\end{center}
\medskip
\medskip
\end{document}